\documentclass[11pt,a4paper,reqno]{amsart} 
\usepackage[applemac]{inputenc}
\usepackage{amsmath}
\usepackage{amssymb}
\usepackage{euscript}
\usepackage{pdfsync}
\usepackage{mathrsfs}
\usepackage{wasysym}
\usepackage[all]{xy}
\usepackage{color}
\usepackage[colorlinks=true,linktocpage=true,pagebackref=false, citecolor=black,linkcolor=black]{hyperref}
\newtheorem{thm}{Theorem}[section]
\newtheorem{mthm}{Theorem}
\newtheorem{lem}[thm]{Lemma}

\newtheorem{cor}[thm]{Corollary}

\theoremstyle{definition}

\theoremstyle{remark}
\newtheorem{remark}[thm]{Remark}

\newtheorem{example}[thm]{Example}

\numberwithin{equation}{section}

{\em}

{\em}

\newcounter{substep}
\def\thesubstep{\arabic{substep}}

\newenvironment{substeps}[1]{%
\refstepcounter{substep}\noindent{ (\ref{#1}.\thesubstep)\ }\ }%
{\em}

\newcommand{\R}{{\mathbb R}}

\newcommand{\psd}{{\mathcal P}}

\newcommand{\J}{{\mathcal J}}
\newcommand{\Cont}{{\mathcal C}}

\newcommand{\gtp}{{\mathfrak p}}\newcommand{\gtq}{{\mathfrak q}}
\newcommand{\gtm}{{\mathfrak m}}\newcommand{\gtn}{{\mathfrak n}}
\newcommand{\gta}{{\mathfrak a}}
\newcommand{\gtP}{{\mathfrak P}}

\newcommand{\Bb}{{\EuScript B}}

\newcommand{\Ff}{{\EuScript F}}

\newcommand{\im}{\operatorname{im}}
\newcommand{\qf}{\operatorname{qf}}

\newcommand{\dist}{\operatorname{dist}}

\newcommand{\hgt}{\operatorname{ht}}

\newcommand{\tr}{\operatorname{tr}}

\newcommand{\id}{\operatorname{id}}

\newcommand{\cl}{\operatorname{Cl}}

\newcommand{\diam}{{\text{\tiny$\displaystyle\diamond$}}}
\newcommand{\gtmd}{\operatorname{\gtm^{\diam}\hspace{-1.5mm}}}
\newcommand{\gtnd}{\operatorname{\gtn^{\diam}\hspace{-1.5mm}}}

\newcommand{\x}{{\tt x}}\newcommand{\y}{{\tt y}}

\newcommand{\veps}{\varepsilon}

\newcommand{\ol }{\overline}
\newcommand{\qq}[1]{\langle{#1}\rangle}

\begin{document}

\title[On the Krull dimension of rings of semialgebraic functions]{On the Krull dimension\\ of rings of semialgebraic functions}

\author{Jos\'e F. Fernando}
\author{J.M. Gamboa}
\address{Departamento de \'Algebra, Facultad de Ciencias Matem\'aticas, Universidad Complutense de Madrid, 28040 MADRID (SPAIN)}
\email{josefer@mat.ucm.es, jmgamboa@mat.ucm.es}
\thanks{Authors supported by Spanish GAAR MTM2011-22435}

% \date is required; it is the date received by the editor.
\date{June 1st, 2013}

\begin{abstract}
Let $R$ be a real closed field and let ${\mathcal S}(M)$ be the ring of (continuous) semialgebraic functions on a semialgebraic set $M\subset R^n$ and let ${\mathcal S}^*(M)$ be its subring of bounded semialgebraic functions. In this work we introduce the concept of \em semialgebraic depth \em of a prime ideal $\gtp$ of ${\mathcal S}(M)$ in order to provide an elementary proof of the finiteness of the Krull dimension of the rings ${\mathcal S}(M)$ and ${\mathcal S}^*(M)$, inspired in the classical way of doing to compute the dimension of a ring of polynomials on a complex algebraic set and without involving the sophisticated machinery of real spectra. We also show that $\dim{\mathcal S}(M)=\dim{\mathcal S}^*(M)=\dim M$ and we prove that in both cases the height of a maximal ideal corresponding to a point $p\in M$ coincides with the local dimension of $M$ at $p$. In case $\gtp$ is a prime \em $z$-ideal \em of ${\mathcal S}(M)$, its semialgebraic depth coincides with the transcendence degree over $R$ of the real closed field $\qf({\mathcal S}(M)/\gtp)$. 
\end{abstract}

\maketitle

\section*{Introduction}\label{s1}

A subset $M\subset\R^n$ is \em semialgebraic \em when it has a description by a finite boolean combination of polynomial equations and inequalities, which we will call a \em semialgebraic \em description. A (continuous) map $f:M\to N$ is \em semialgebraic \em if its graph is a semialgebraic set (in particular $M$ and $N$ are semialgebraic). In case $N=R$, we say that $f:M\to R$ is a \em semialgebraic function\em. The sum and product defined pointwise endow the set ${\mathcal S}(M)$ of semialgebraic functions on $M$ with a natural structure of unital commutative ring. It is obvious that the subset ${\mathcal S}^*(M)$ of bounded semialgebraic functions on $M$ is a real subalgebra of ${\mathcal S}(M)$. For the time being we denote by ${\mathcal S}^{\diam}(M)$, indistinctly, either ${\mathcal S}(M)$ or ${\mathcal S}^*(M)$ in case the involved statements or arguments are valid for both rings simultaneously. For instance, if $p\in M$, we will denote by $\gtmd_p$ the maximal ideal of all functions in ${\mathcal S}^{\diam}(M)$ vanishing at $p$. For each $f\in{\mathcal S}^{\diam}(M)$ and each semialgebraic subset $N\subset M$, we denote $Z_N(f):=\{x\in N:\, f(x)=0\}$. In case $N=M$, we say that $Z(f):=Z_M(f)$ is the \em zeroset \em of $f$.

As it is well-known the rings ${\mathcal S}^{\diam}(M)$ are particular cases of the so-called \em real closed rings \em introduced by Schwartz \cite{s0} in the '80s of the last century. The theory of real closed rings has been deeply developed until now in a fruitful attempt to establish new foundations for semi-algebraic geometry with relevant interconnections with model theory, see the results of Cherlin-Dickmann \cite{cd1,cd2}, Schwartz \cite{s0,s1,s2,s3}, Schwartz with Prestel, Madden and Tressl \cite{ps,sm,scht} and Tressl \cite{t0,t1,t2}. We refer the reader to \cite{s1} for a ring theoretic analysis of the concept of real closed ring. Moreover, this theory, which vastly generalizes the classical techniques concerning the semi-algebraic spaces of Delfs-Knebusch (see \cite{dk2}), provides a powerful machinery to approach problems about certain rings of real valued functions, and contributes to achieve a better understanding of the algebraic properties of such rings and the topological properties of their spectra. We highlight some of them: (1) real closed fields; (2) rings of real-valued continuous functions on Tychonoff spaces; (3) rings of semi-algebraic functions on semi-algebraic subsets of $R^n$; and more generally (4) rings of definable continuous functions on definable sets in o-minimal expansions of fields.

In this work we provide an elementary geometric proof of the fact that the ring ${\mathcal S}^{\diam}(M)$ has finite Krull dimension and we show that $\dim{\mathcal S}^{\diam}(M)$ equals the dimension of $M$. Despite they are neither noetherian nor enjoy primary decomposition results, these rings are closer to polynomial rings than to the classical rings of continuous functions. For instance, the Lebesgue dimension of $R$ is $1$ (see \cite[16F]{gj}) while the Krull dimension of the ring $\Cont( R)$ of real valued continuous functions on $R$ is infinite (see \cite[14I]{gj}). 

Recall that an ideal $\gta$ of ${\mathcal S}(M)$ is a \em $z$-ideal \em if given two semialgebraic functions $f\in\gta$ and $g\in{\mathcal S}(M)$ such that $Z(f)\subset Z(g)$ it holds that $g\in\gta$. Notice that each $z$-ideal is a real ideal. It follows from \cite[2.6.6]{bcr} that if $M$ is locally closed the $z$-ideals coincide with the radical ideals; in particular, all prime ideals are $z$-ideals. Of course, this is not further true if $M$ is not locally closed. 

In the polynomial context over an algebraically closed field $C$, Hilbert's Nullstellensatz assures that the radical ideals of $C[\x]:=C[\x_1,\ldots,\x_n]$ coincide with the zero ideals of subsets of $C^n$. Thus, when handling chains of prime ideals in rings of semialgebraic functions, we are nearer to the polynomial case over an algebraically closed field $C$ than to the one having coefficients in a real closed field $R$, where the longest chains appear only when dealing with the so called \em real prime ideals \em (see \cite[\S4.1]{bcr}). This is why we follow similar guidelines to those involved to prove that the Krull dimension of a ring of polynomial functions on an algebraic set $Z\subset C^n$ coincides with the dimension of $Z$. Namely, the clue is the following: \em if $\gtP_1\subsetneq\gtP_2$ are two prime ideals of $C[\x]$, the dimension of the zeroset of $\gtP_2$ is strictly smaller to that of $\gtP_1$\em; hence, the dimension is the invariant that allows to bound the number of possible jumps in a chain of prime ideals. 

Nevertheless, it is well known that the common zero set $Z(\gtp)$ of the semialgebraic functions in a prime ideal $\gtp$ of ${\mathcal S}(M)$ is either empty or a point; hence, it has no sense to work with its dimension. We substitute it by the \em semialgebraic depth of $\gtp$ \em that we define as:
$$
{\tt d}_M(\gtp):=\min\{\dim(Z(f)):\ f\in\gtp\}.
$$
Of course, in the polynomial case the corresponding semialgebraic depth of a prime ideal equals the dimension of the zeroset of the ideal. Our main results are the following.

\begin{mthm}[Dimension]\label{dimension} 
The Krull dimensions of the rings ${\mathcal S}(M)$ and ${\mathcal S}^*(M)$ coincide with the topological dimension of $M$.
\end{mthm}

\begin{mthm}[Local dimension]\label{localdimension}
Let $p\in M$ and let $\gtmd_p$ be the maximal ideal of ${\mathcal S}^{\diam}(M)$ associated to $p$. Then, $\hgt(\gtmd_p)$ equals the local dimension $d$ of $M$ at $p$. Moreover, there is a chain of prime ideals $\gtp_0\subsetneq\cdots\subsetneq\gtp_d:=\gtmd_p$ such that the transcendence degree over $R$ of the real closed field $\qf({\mathcal S}^{\diam}(M)/\gtp_k)$ equals $d-k$. In case ${\mathcal S}^{\diam}(M)={\mathcal S}(M)$ the ideals $\gtp_k$ can be chosen to be a $z$-ideals.
\end{mthm}

Recall that Carral and Coste proved in \cite{cc} the equality $\dim{\mathcal S}(M)=\dim M$ for a locally closed semialgebraic set $M$ (see also \cite{g,s2,s4}) by proving that the real spectrum of ${\mathcal S}(M)$ is homeomorphic to the subset $\widetilde{M}$ of the real spectrum of the ring ${\mathcal P}(M)$ of polynomial functions on $M$ (see \cite[Ch.7]{bcr} for the technicalities concerning the real spectrum). Later on Gamboa-Ruiz extended this equality in \cite{gr} to an arbitrary semialgebraic set, using strong properties of the real spectrum of excellent rings and some crucial results of the theory of real closed rings (see \cite{s2}). As far as we know, the equality $\dim{\mathcal S}^*(M)=\dim M$ was unknown until now and it was also unknown the characterization of the local dimension of a semialgebraic set at one of its points $p$ in terms of the maximal chains of ideals contained in the maximal ideal $\gtmd_p$. 

On the other hand, in the algebraic case it holds that the transcendence degree of the quotient field of the ring of polynomial functions on an irreducible algebraic set $Z$ coincides with the dimension of $Z$ (see \cite[11.25]{am}). In the semialgebraic setting, we prove, by using crucially that ${\mathcal S}(R^n)$ is the real closure of the ring $R[\x]:=R[\x_1,\ldots,\x_n]$, the following analogous result.

\begin{mthm}\label{rmdc}
Let $\gtp\subset{\mathcal S}^{\diam}(M)$ be a prime ideal. Then, 
\begin{itemize}
\item[(i)] The transcendence degree over $R$ of the real closed field $\qf({\mathcal S}^{\diam}(M)/\gtp)$ is finite and upperly bounded by $\dim M$.
\item[(ii)] If $\gtp$ is moreover a prime $z$-ideal, ${\tt d}_M(\gtp)=\tr\deg_R(\qf({\mathcal S}(M)/\gtp))$.
\end{itemize}
\end{mthm}

The article is organized as follow. In Section \ref{s2}, we show that ${\mathcal S}^*(M)$ is the direct limit of the rings ${\mathcal S}(X)$ where $X$ runs over the \em semialgebraic pseudo-compactifications of $M$\em. This fact is the key point to prove the finiteness of the dimension of ${\mathcal S}^*(M)$ and its localization ${\mathcal S}(M)$. In Section \ref{s3}, we study some properties of the semialgebraic depth of a prime ideal of ${\mathcal S}(M)$ and we prove Theorem \ref{rmdc}. Finally, Theorems \ref{dimension} and \ref{localdimension} are proved in Section \ref{s4}.

\section{Semialgebraic pseudo-compactifications}\label{s2}

A \em semialgebraic pseudo-compactification of $M$ \em is a pair $(X,{\tt j})$ constituted by a closed and bounded semialgebraic set $X\subset R^n$ and a semialgebraic embedding ${\tt j}:M\hookrightarrow X$ whose image is dense in $X$. Of course, it holds that ${\mathcal S}(X)={\mathcal S}^*(X)$ since the image of a bounded and closed semialgebraic set under a semialgebraic function is again bounded and closed. The embedding ${\tt j}$ induces an $R$-monomorphism $j^*:{\mathcal S}(X)\hookrightarrow{\mathcal S}^{\diam}(M),\ f\mapsto f\circ{\tt j}$ and we will denote $\gta\cap{\mathcal S}(X):={\tt j}^{*,-1}(\gta)$ for every ideal $\gta$ of ${\mathcal S}^{\diam}(M)$.

Sometimes it will be useful to assume that the semialgebraic set $M$ is bounded. Namely, the semialgebraic homeomorphism between the open ball $\Bb_n$ of center $0$ and radius $1$ and $R^n$
$$
h:\Bb_n\to R^n\ x\mapsto\frac{x}{\sqrt{1-\|x\|^2}},
$$
induces an $R$-isomorphism ${\mathcal S}(M)\to {\mathcal S}(h^{-1}(M)),\,f\mapsto f\circ h$. Thus, we may always assume that $M$ is bounded and in particular that the closure $\cl(M)$ of $M$ (in $R^n$) is a semialgebraic pseudo-compactification of $M$.

\subsection{Properties of the semialgebraic pseudo-compactifications.}\label{cita}
The following properties are decisive:

\vspace{2mm}
\begin{substeps}{cita}\label{a}
\em
For each finite family $\Ff:=\{f_1,\ldots,f_r\}\subset{\mathcal S}^*(M)$ there exist a semialgebraic pseudo-compacti\-fication $(X,{\tt j}_{\Ff})$ of $M$ and semialgebraic functions $F_1,\ldots,F_r\in{\mathcal S}(X)$ such that $f_i=F_i\circ{\tt j}_{\Ff}$. 
\em
\end{substeps}

\vspace{2mm}
Indeed, we may assume that $M$ is bounded. Now, consider $X:=\cl({\rm graph}(f_1,\ldots,f_r))$, ${\tt j}_\Ff:M\hookrightarrow X,\ x\mapsto(x,f_1(x),\ldots,f_r(x))$ and $F_i:=\pi_{n+i}|_{X}$, where $\pi_{n+i}:R^{n+r}\to R,\ x:=(x_1,\ldots,x_{n+r})\mapsto x_{n+i}$ for $i=1,\ldots,r$. 

\vspace{2mm}
\begin{substeps}{cita}\label{longitud}
\em
Given a chain of prime ideals $\gtp_0\subsetneq\cdots\subsetneq \gtp_r$ of ${\mathcal S}^*(M)$ there is a semialgebraic compactification $(X,{\tt j})$ of $M$ such that the prime ideals $\gtq_i:=\gtp_i\cap{\mathcal S}(X)$ constitute a chain $\gtq_0\subsetneq\cdots\subsetneq\gtq_r$ in ${\mathcal S}(X)$. 
\em
\end{substeps}

\vspace{2mm}
Indeed, it is enough to pick $f_i\in\gtp_i\setminus \gtp_{i-1}$ for $1\leq i \leq r$ and to consider the semialgebraic pseudo-compactification of $M$ provided by (\ref{cita}.\ref{a}) for the family $\Ff:=\{f_1,\ldots,f_r\}$.

\vspace{2mm}
\begin{substeps}{cita}
Let ${\mathfrak F}_M$ be the collection of all the semialgebraic pseudo-compactifications of $M$. Given $(X_1,{\tt j}_1),(X_2,{\rm j}_2)\in{\mathfrak F}_M$ we say that $(X_1,{\tt j}_1)\preccurlyeq(X_2,{\rm j}_2)$ if and only if there is a (unique) continuous surjective map $\rho:X_2\to X_1$ such that $\rho\circ{\tt j}_2={\tt j}_1$; the uniqueness of $\rho$ follows because $\rho|_M={\tt j}_1\circ({\tt j}_2|_M)^{-1}$ and $M$ is dense in $X_i$. We claim that: \em $({\mathfrak F}_M,\preccurlyeq)$ is a directed set.\em
\end{substeps}

\vspace{2mm}
Indeed, let $(X_1,{\tt j}_1),(X_2,{\tt j}_2)\in{\mathfrak F}_M$ and consider the semialgebraic map 
$$
{\tt j}_3:M\to X_3:=\cl_{R^{n+p}}(({\tt j}_1,{\tt j}_2)(M)),\ x\mapsto({\tt j}_1(x),{\tt j}_2(x));
$$ 
notice that $(X_3,{\tt j}_3)\in{\mathfrak F}_M$ and $(X_1,{\tt j}_1)\preccurlyeq(X_3,{\tt j}_3)$ and $(X_2,{\tt j}_2)\preccurlyeq(X_3,{\tt j}_3)$.

\vspace{2mm}
\begin{substeps}{cita}
We have a collection of rings $\{{\mathcal S}(X)\}_{(X,{\tt j})\in{\mathfrak F}_M}$ and $R$-monomorphisms 
$$
\rho_{X_1,X_2}^*:{\mathcal S}(X_1)\to{\mathcal S}(X_2),\ f\mapsto f\circ\rho
$$ 
for $(X_1,{\tt j}_1)\preccurlyeq(X_2,{\tt j}_2)$ such that 
\begin{itemize}
\item $\rho_{X_1,X_1}^*=\id$, and
\item $\rho_{X_1,X_3}^*=\rho_{X_2,X_3}^*\circ\rho_{X_1,X_2}^*$ if $(X_1,{\tt j}_1)\preccurlyeq(X_2,{\tt j}_2)\preccurlyeq(X_3,{\tt j}_3)$.
\end{itemize}

We conclude that: \em The ring ${\mathcal S}^*(M)$ is the direct limit of the directed system $\qq{{\mathcal S}(X),\rho_{X_1,X_2}^*}$ together with the homomorphisms ${\tt j}^*:{\mathcal S}(X)\hookrightarrow{\mathcal S}^*(M)$, where $(X,{\tt j})\in{\mathfrak F}_M$\em. We will write ${\mathcal S}^*(M)=\displaystyle\lim_{\longrightarrow}{\mathcal S}(X)$. 
\end{substeps}

\vspace{2mm}
\begin{substeps}{cita}\label{inequality}
On the other hand, \em the ring ${\mathcal S}(M)$ is the localization ${\mathcal S}^*(M)_{{\mathcal W}(M)}$ of ${\mathcal S}^*(M)$ at the multiplicative set ${\mathcal W}(M)$ of those functions $f\in{\mathcal S}^*(M)$ such that $Z_M(f)=\varnothing$\em. In particular, \em if $\gtp$ is a prime ideal of ${\mathcal S}^*(M)$ that do not intersect ${\mathcal W}(M)$, then $\qf({\mathcal S}^*(M)/\gtp)=\qf({\mathcal S}(M)/\gtp{\mathcal S}(M))$\em.
\end{substeps}

\vspace{2mm}
This is pretty evident since each function $f\in{\mathcal S}(M)$ is the quotient $f=g/h$, where $g:=f/(1+f^2)\in{\mathcal S}^*(M)$ and $h:=1/(1+f^2)\in{\mathcal W}(M)$.

\vspace{2mm}
\begin{substeps}{cita}\label{inequalities}
We conclude from (\ref{cita}.\ref{longitud}) and (\ref{cita}.\ref{inequality}) that:
\begin{equation*}\label{inequalities0}
\dim{\mathcal S}(M)\leq\dim{\mathcal S}^*(M)\leq\sup_{(X,{\tt j})\in{\mathfrak F}_M}\{\dim{\mathcal S}(X)\}.
\end{equation*}
\end{substeps}

\section{Semialgebraic depth of a prime ideal}\label{s3}

The purpose of this section is to study some properties of the semialgebraic depth of a prime ideal of ${\mathcal S}(M)$ and to prove Theorem \ref{rmdc}. A crucial fact when dealing with the ring of semialgebraic functions on a semialgebraic set $M$ is the fact that every closed semialgebraic subset $Z$ of $M$ is the zeroset $Z(h)$ of a (bounded) semialgebraic function $h$ on $M$; take for instance, $h:=\min\{1,\dist(\cdot,Z)\}\in {\mathcal S}^*(M)$.

\begin{lem}\label{ht0}
Let $\gtp,\gtq$ be two prime $z$-ideals of ${\mathcal S}(M)$ such that $\gtq\subsetneq\gtp$. Then, ${\tt d}_M(\gtp)<{\tt d}_M(\gtq)$.
\end{lem}
\begin{proof}
Suppose, by way of contradiction, that there are two prime $z$-ideals $\gtp,\gtq\subset{\mathcal S}(M)$ such that $\gtq\subsetneq\gtp$ and ${\tt d}_M(\gtp)={\tt d}_M(\gtq)$. Let $g\in\gtq$ be such that ${\tt d}_M(\gtq)=\dim Z_M(g)$. We may assume that the Zariski closure $Y$ of $Z_M(g)$ is irreducible.

\vspace{2mm}
Indeed, for each irreducible component $Y_i$ of $Y$ let $P_i$ be a polynomial such that $Y_i=Z_{R^n}(P_i)$. The product $P:=\prod_iP_i$ satisfies $Z_M(g)\subset Y=Z_{R^n}(P)$ and since $\gtq$ is a $z$-ideal, $P\in \gtq$; hence, there is some index $i$ such that $P_i\in\gtq$. Now, picking $g':=g^2+P_i^2\in\gtq$ from the beginning instead of $g$, we have ${\tt d}_M(\gtq)=\dim Z_M(g')$ and the Zariski closure $Y_i$ of $Z_M(g')$ is irreducible.

\vspace{2mm}
Next, we choose  $f\in\gtp\setminus\gtq$ and define $f':=f^2+g^2$. Clearly, $f'\in\gtp\setminus\gtq$ and
$$
\dim Z_M(f')\leq \dim Z_M(g)={\tt d}_M(\gtq)={\tt d}_M(\gtp)\leq \dim Z_M(f');
$$
hence, $\dim Z_M(f')=\dim Z_M(g)=\dim Y$. Thus, since $Z_M(f')\subset Z_M(g)\subset Y$ and the irreducible algebraic set $Y$ has the same dimension as $Z_M(f')$, we deduce that $Y$ is the Zariski closure of $Z_M(f')$. Denote $T:=Y\setminus Z_M(f')$ and let $h\in{\mathcal S}(R^n)$ be such that $\cl(T)=Z_{R^n}(h)$. As before, we denote by $P$ a polynomial equation of $Y$ and we have
$$
Z_M(P)=Y\cap M=(\cl(T)\cap M)\cup Z_M(f')=Z_M(f'h).
$$
Since $P\in\gtq$ and $\gtq$ is a $z$-ideal, $f'h\in\gtq$. Thus, $\gtq$ being prime and $f'\not\in\gtq$, we deduce that $h\in\gtq\subset\gtp$. Therefore, $h':={f'}^2+h^2\in\gtp$ and so $\dim Z_M(h')\geq {\tt d}_M(\gtp)$. Now,
$$
Z_M(h')=Z_M(f')\cap Z_M(h)\subset Z_M(f')\cap\cl(T)=\cl(T)\setminus T;
$$
hence, using \cite[2.8.13]{bcr}, we deduce that
$$
{\tt d}_M(\gtp)\leq\dim(Z_M(h'))\leq\dim(\cl(T)\setminus T)<\dim T\leq\dim Y={\tt d}_M(\gtp),
$$
a contradiction. We conclude ${\tt d}_M(\gtp)<{\tt d}_M(\gtq)$, as required. 
\end{proof}

\begin{remark}
The previous result \ref{ht0} is false if either $\gtp$ or $\gtq$ is not a $z$-ideal. Indeed, consider the triangle $M:=\{(x,y)\in R^2:\, 0<y\leq x\leq1\}\cup\{p=(0,0)\}$ and the prime ideal of ${\mathcal S}(M)$:
$$
\gtp:=\{f\in{\mathcal S}(M):\, \exists\,\veps>0\,|\, f\text{ extends continuously by $0$ to } M\cup((0,\veps]\times\{0\})\}\subsetneq\gtm_p.
$$
A straightforward computation shows that ${\tt d}_M(\gtp)={\tt d}_M(\gtm_p)=0$.
\end{remark}

\begin{cor}\label{ht}
Assume that $M$ is a $d$-dimensional locally closed semialgebraic set. Then, for every prime ideal $\gtp$ of ${\mathcal S}(M)$ it holds ${\tt d}_M(\gtp)+\hgt(\gtp)\leq d$. In particular, the Krull dimension $\dim{\mathcal S}(M)\leq d$.
\end{cor}
\begin{proof}
As $M$ is locally closed, all the prime ideals of ${\mathcal S}(M)$ are $z$-ideals. Note that given a chain of prime ideals $\gtp_0\subsetneq\cdots\subsetneq\gtp_r$ in ${\mathcal S}(M)$, we get, by Lemma \ref{ht0}, ${\tt d}_M(\gtp_r)<\cdots<{\tt d}_M(\gtp_0)\leq d,$ and so, ${\tt d}_M(\gtp_r)\leq d-r$. Whence, $\hgt(\gtp)+{\tt d}_M(\gtp)\leq d$ for every prime ideal $\gtp\subset{\mathcal S}(M)$, as wanted.
\end{proof}
 
Next, we proceed with the proof of Theorem \ref{rmdc}. Recall first that every commutative ring $A$ has a so called \em real closure \em ${\rm rcl}(A)$ and this is unique up to a unique ring homomorphism over $A$. This means that ${\rm rcl}(A)$ is a real closed ring and there is a (not necessarily injective) ring homomorphism $\gamma:A\to{\rm rcl}(A)$ such that for every ring homomorphism $\psi:A\to B$ to some other real closed ring $B$, there is a unique ring homomorphism $\ol{\psi}:{\rm rcl}(A)\to B$ with $\psi=\ol{\psi}\circ\gamma$. For example, the real closure of the polynomial ring $R[\x]:=R[\x_1,\ldots,\x_n]$ is the ring ${\mathcal S}(R^n)$ of semi-algebraic functions on $R^n$. More generally, if $Z\subset R^n$ is an algebraic set, then ${\mathcal S}(Z)$ is the real closure of the ring $\psd(Z):=R[\x]/\J(Z)$ of polynomials on an algebraic set $Z\subset R^n$, where $\J(Z):=\{f\in R[\x]:\ Z\subset Z(f)\}$.

In particular, if $\varphi:\psd(Z)\to F$ is an $R$-homomorphism into a real closed field $F$, there is a unique $R$-homomorphism $\ol{\varphi}:{\mathcal S}(Z)\to F$ such that $\varphi=\ol{\varphi}\circ\gamma$ where $\gamma:\psd(Z)\hookrightarrow{\mathcal S}(Z)$ is the natural inclusion. Unfortunately, this result extends no more to arbitrary semialgebraic sets, even if they are closed.

\begin{example}\label{ptd}
Consider the semialgebraic subsets $M:=\{x^2+y^2<1\}$ and $K:=\{x^2+y^2\leq1\}$ of $R^2$. We have that $R[\x,\y]=\psd(R^2)=\psd(K)=\psd(M)$, hence, the real closure of this ring is ${\mathcal S}(R^2)$ but in the first row of the following diagram
$$
\xymatrix{
{\mathcal S}(R^2)\ar@{->>}[r]&{\mathcal S}(K)\ar@{^(->}[r]&{\mathcal S}^*(M)\ar@{^(->}[r]&{\mathcal S}(M)\\
R[\x,\y]\ar@{^(->}[u]\ar@{=}[r]&\psd(K)\ar@{^(->}[u]\ar@{=}[r]&\psd(M)\ar@{^(->}[u]\ar@{^(->}[ru]
}
$$
none of the involved homomorphisms is bijective.\qed 
\end{example}

\begin{proof}[Proof of Theorem \em\ref{rmdc}] 
We develope this proof in three steps:

\vspace{2mm}
\noindent{\bf Step 1.} We approach first the case when $M=X$ is closed and bounded: \em Let $X$ be a closed and bounded semialgebraic set and let $\gtp$ be a prime ideal of ${\mathcal S}(X)$. Then,
$$
{\tt d}_X(\gtp)=\tr\deg_R(\qf({\mathcal S}(X)/\gtp)).
$$
\em 

Indeed, let $Y$ be the Zariski closure of $X$ and write $\gtP:=\gtp\cap\psd(Y)$. Notice that $\gtP$ is a real prime ideal and so the irreducible algebraic set $Z:=Z(\gtP)\subset Y$ satisfies $\psd(Z)\equiv\psd(Y)/\gtP$.

Since $\tr\deg_R(\qf(\psd(Z)))=\dim Z$, it is enough to check that ${\tt d}_{X}(\gtp)=\dim Z$ and that the field extension $\qf(\psd(Z))=\qf(\psd(Y)/\gtP)\hookrightarrow\qf({\mathcal S}(X)/\gtp)$ is algebraic. For the last assertion it suffices that $\qf({\mathcal S}(X)/\gtp)$ is the real closure of $\qf(\psd(Z))$ with respect to the ordering induced by $\qf({\mathcal S}(X)/\gtp)$ on $\qf(\psd(Z))$.

Indeed, let us check first that ${\tt d}_{X}(\gtp)=\dim Z$, the inequality ${\tt d}_{X}(\gtp)\leq\dim Z$ being clear. Suppose that there is $f\in\gtp$ such that $\dim(Z(f))<\dim Z$. Then, the Zariski closure $Z'$ of $Z(f)$ has its same dimension. Let $h'\in R[\x]$ be a polynomial equation of $Z'$ and observe that $h'\in\gtp$, because $Z(f)\subset Z(h')$ (use \cite[2.6.6]{bcr}); hence, $h'\in\gtP:=\gtp\cap\psd(Y)$ and so $Z\subset Z'$, which contradicts the fact that $\dim Z'<\dim Z$. 

Next, we prove that $\qf({\mathcal S}(X)/\gtp)$ is the real closure of $\qf(\psd(Z))$ with respect to the ordering induced by $\qf({\mathcal S}(X)/\gtp)$ on $\qf(\psd(Z))$. Consider the diagram
$$
\begin{array}{cccc}
f&\longmapsto&f|_{X\cap Z}\\[4pt]
{\mathcal S}(X)&\overset{\theta_1}{\longrightarrow}&{\mathcal S}(X\cap Z)&g|_{X\cap Z}\\
&&{\text{\footnotesize$\displaystyle\theta_2$}}\uparrow&\uparrow\\
&&{\mathcal S}(Z)&g
\end{array}
$$
Since $X\cap Z$ is closed in $X$ and $Z$, we deduce, by \cite{dk}, that $\theta_1$ and $\theta_2$ are epimorphisms; hence, $\gtq:=\theta_2^{-1}(\theta_1(\gtp))$ is a prime ideal. By the first isomorphism theorem,
\begin{multline*}
{\mathcal S}(X)/\gtp\cong{\mathcal S}(X\cap Z)/\theta_1(\gtp)\cong{\mathcal S}(Z)/\gtq\ \text{ and so }\\ \qf({\mathcal S}(X)/\gtp)\cong\qf({\mathcal S}(X\cap Z)/\theta_1(\gtp))\cong\qf({\mathcal S}(Z)/\gtq).
\end{multline*}
Since $\gtP=\gtp\cap\psd(Y)$ and $\psd(Z)=\psd(Y)/\gtP$, we deduce $\gtq\cap\psd(Z)=\{0\}$ and so $\psd(Z)\hookrightarrow{\mathcal S}(Z)/\gtq$. Thus, it is enough to assure that $\qf({\mathcal S}(Z)/\gtq)$ is the real closure of $\qf(\psd(Z))$ with respect to the ordering induced by $\qf({\mathcal S}(Z)/\gtq)$ on $\qf(\psd(Z))$; but this follows from \cite[\S I.4, p.27]{s2} and we are done.

\vspace{2mm}
\noindent{\bf Step 2.} \em Let $\gtp$ be a prime ideal of ${\mathcal S}^*(M)$. Then, there is a semialgebraic pseudo-compactification $(X,{\tt j})$ of $M$ such that 
$$
\qf({\mathcal S}(X)/(\gtp\cap{\mathcal S}(X)))=\qf({\mathcal S}^*(M)/\gtp).
$$
\em We refer to $(X,{\tt j})$ as a \em brimming semialgebraic pseudo-compactification of $M$ for $\gtp$\em. 

By (\ref{cita}.\ref{inequality}), it is enough to consider the case ${\mathcal S}^*(M)={\mathcal S}^*(M)$ and we assume moreover that $M$ is bounded. Consider the real closed field $F:=\qf({\mathcal S}^*(M)/\gtp)$ and the $R$-homomorphism
$$
\varphi:{\mathcal S}^*(M)\to{\mathcal S}^*(M)/\gtp\hookrightarrow F.
$$
For each finite subset $\Ff:=\{f_1,\ldots,f_r\}\subset{\mathcal S}^*(M)$ consider the semialgebraic pseudo-compactifi\-cation $(X_\Ff,{\tt j}_\Ff)$ constructed in (\ref{cita}.\ref{a}) for $\Ff$ an let $\varphi_\Ff:=\varphi\circ{\tt j}_\Ff^*:{\mathcal S}(X_\Ff)\to F$. Denote $\gtp_\Ff:=\ker\varphi_\Ff=\gtp\cap{\mathcal S}(X_\Ff)$ and ${\tt d}:=\max_\Ff\{{\tt d}_{X_\Ff}(\gtp_\Ff))\}$ where $\Ff$ runs over all the finite subsets of ${\mathcal S}^*(M)$. We claim that: \em The transcendence degree over $R$ of $F$ equals ${\tt d}$\em.

Indeed, fix a finite family $\Ff_0$ such that ${\tt d}_{X_{\Ff_0}}(\gtp_{\Ff_0})={\tt d}$, and denote $X_0:=X_{\Ff_0}$; clearly, $F_0:=\qf({\mathcal S}(X_0)/(\gtp\cap{\mathcal S}(X_0)))\subset F$. Since both are real closed fields, to prove that they are equal it is enough to show that $F$ is an algebraic extension of $F_0$, that is, $f+\gtp$ is algebraic over $F_0$ for every $f\in{\mathcal S}^*(M)\setminus\gtp$. Fix such an $f$ and consider the set $\Ff_1:=\Ff_0\cup\{f\}$ and the semialgebraic pseudo-compactification $X_1:=X_{\Ff_1}$. The projection onto all the coordinates except the last one induces a surjective semialgebraic map $\rho:X_1\to X_0$, whose restriction to $M$ is a semialgebraic homeomorphism. This map induces the $R$-monomorphism ${\mathcal S}(X)\hookrightarrow {\mathcal S}(X_1),\,h\mapsto h\circ\rho$. In this way, we have the following diagrams of ring monomorphisms
$$
\xymatrix{
{\mathcal S}(X)\ar@{^{(}->}[r]\ar@{^{(}->}[rd]&{\mathcal S}(X_1)\ar@{^{(}->}[d]\\
&{\mathcal S}^*(M)}
\qquad
\xymatrix{
{\mathcal S}(X)/(\gtp\cap{\mathcal S}(X))\ar@{^{(}->}[r]\quad \ar@{^{(}->}[rd]&{\mathcal S}(X_1)/(\gtp\cap{\mathcal S}(X_1))\ar@{^{(}->}[d]\\
&{\mathcal S}^*(M)/\gtp}
$$
and so $F_0\subset F_1:=\qf({\mathcal S}(X_1)/(\gtp\cap{\mathcal S}(X_1)))\subset F$. Thus, to see that $f+\gtp$ is algebraic over $F_0$ it is enough that $f+(\gtp\cap{\mathcal S}(X_1))$ is algebraic over $F_0$, and for that it suffices that the trancendence degree over $R$ of $F_0$ and $F_1$ coincide. Since $F_0\subset F_1$, it follows from Step 1 that
$$
\tr\deg_R(F_0)\leq\tr\deg_R(F_1)={\tt d}_{X_1}(\gtp_{{\mathfrak F}_1})\leq{\tt d}={\tt d}_{X_0}(\gtp_{{\mathfrak F}_0})=\tr\deg_\R(F_0);
$$
hence, $F=F_0$, as wanted.

\vspace{2mm}
\noindent{\bf Step 3.} \em Proof of the statement\em. Assertion (i) follows straightforwardly from Steps 1 and 2. To prove (ii) pick a brimming semialgebraic pseudo-compactification $(X,{\tt j})$ of $M$ for $\gtp$; we identify $M$ with ${\tt j}(M)$. By Step 1, 
$$
{\tt d}_X(\gtp\cap{\mathcal S}(X))=\tr\deg_R(\qf({\mathcal S}(X)/(\gtp\cap{\mathcal S}(X))))=\tr\deg_R(\qf({\mathcal S}(M)/\gtp)).
$$
Thus, it only remains to check that ${\tt d}_X(\gtp\cap{\mathcal S}(X))={\tt d}_M(\gtp)$, being clear the inequality ${\tt d}_M(\gtp)\leq{\tt d}_X(\gtp\cap{\mathcal S}(X))$. For the converse inequality, let $f\in\gtp$ be such that ${\tt d}_M(\gtp)=\dim Z_M(f)$ and pick a function $g\in{\mathcal S}(X)$ such that $\cl_X(Z_M(f))=Z_X(g)$. As $\gtp$ is a $z$-ideal, $g|_M\in\gtp$ and so $g\in\gtp\cap{\mathcal S}(X)$. Therefore,
$$
{\tt d}_M(\gtp)=\dim Z_M(f)=\dim Z_X(g)\geq{\tt d}_X(\gtp\cap{\mathcal S}(X)), 
$$
as wanted.
\end{proof}

\section{Krull dimension of rings of semialgebraic functions}\label{s4}

In this section we prove Theorems \ref{dimension} and \ref{localdimension}. Before that, we need some preliminary results.

\begin{lem}\label{subset}
Let $p\in M$ and denote by $d$ the local dimension of $M$ at $p$. Then, there is a semialgebraic embedding $h:[0,1]^d\hookrightarrow M$ such that $h(0)=p$.
\end{lem}
\begin{proof}
We assume that $M$ is bounded and so its closure $X:=\cl(M)$ is bounded and closed. Let ${\mathcal C}:=\{C_1,\ldots,C_s\}$ be a semialgebraic cellular decomposition of $X$ such that $M$ and the closure of each cell are (finite) unions of cells, see \cite[9.1.11-12]{bcr}. Denote by $g_i:[0,1]^{\dim C_i}\to X$ the characteristic map of the cell $C_i$ and recall that the restriction $g_i|_{(0,1)^{\dim C_i}}:(0,1)^{\dim C_i}\to C_i$ is a semialgebraic homeomorphism onto $C_i$. Clearly, 
$$
d:=\dim_pM=\max\{\dim C_i:\, p\in\cl_{M}(C_i),\ C_i\subset M\}.
$$
Without loss of generality, suppose that $p\in\cl_{M}(C_1)$, $C_1\subset M$ and $\dim C_1=d$. Of course, $\cl_{M}(C_1)\subset g_1([0,1]^d)$ and we pick a point $q\in[0,1]^d$ such that $g_1(q)=p$. Let us construct a semialgebraic set $T_1\subset(0,1)^d\cup\{q\}$ and a semialgebraic homeomorphism $h_1:[0,1]^d\to T_1$ mapping the origin to $q$. 

Indeed, after reordering the variables and applying changes of coordinates of the type $x_i\mapsto 1-x_i$, if necessary, we may assume that $q=(0,\ldots,0,q_{r+1},\ldots,q_d)$ where $0<q_j<1$. Consider the semialgebraic set
$$
S:=\{q\}\cup\big((0,1)^r\times(q_{r+1},1)\times\cdots\times(q_d,1)\big)\subset\{q\}\cup(0,1)^d,
$$
and the semialgebraic homeomorphism 
$$
f:=(f_1,\ldots,f_d):Q:=\{0\}\cup\{x_1>0,\ldots,x_d>0\}\subset R^d\to S,
$$ 
where $f_i(x)=q_i+\frac{(1-q_i)x_i}{1+x_i}$ for $i=1,\ldots,d$. Observe that $f(0)=q$ and consider the $d$-dimensional bounded and closed affine parallelepiped $T_0:=\{\sum_{i=1}^d\lambda_i{\tt p}_i:\, 0\leq\lambda_i\leq1\}\subset Q$ generated by the origin and the points ${\tt p}_i:=(1/3,\ldots,1/3,1,1/3,\ldots,1/3)$ all whose coordinates are $1/3$ except for the $i$-th which equals $1$, for $i=1,\ldots,d$. It is clear that $T_0\subset Q$ and the map $f|_{T_0}:T_0\to f(T_0)=:T_1$ is a semialgebraic homeomorphism. Of course, $[0,1]^d$ and $T_0$ are semialgebraically homeomorphic via the semialgebraic homeomorphism $h_0:[0,1]^d\to T_0,\ \lambda:=(\lambda_1,\ldots,\lambda_d)\mapsto\sum_{i=1}^d\lambda_i{\tt p}_i$, which satisfies $h_0(0)=0$. Henceforth, the semialgebraic set $T_1\subset\{q\}\cup(0,1)^d$ and the map $h_1:=(f|_{T_0})\circ h_0:[0,1]^d\to T_1$ are the ones we are looking for. 

Finally, choosing $h:=g_1\circ h_1:[0,1]^d\hookrightarrow M$, we are done.
\end{proof}

The following example provide the clue to construct chains of prime ideals of maximal lenght.

\begin{example}\label{excrucial}
(i) Let $X:=[0,1]^n$ and define $\gtp$ as the set of all semialgebraic functions $f\in{\mathcal S}(X)$ such that for each semialgebraic triangulation $(K,\Phi)$ of $X$ compatible with $Z(f)$ there is an $n$-dimensional simplex $\sigma\in K$ such that $\Phi(\sigma)\subset Z(f)$ and for each $d=0,\ldots,n$ there is a $d$-dimensional face $\tau_d$ of $\sigma$ so that $\Phi(\tau_d)\subset\{\x_{d+1}=0,\ldots,\x_n=0\}$. We will call $\sigma$ an \em indicator simplex for $f$\em. We claim that: \em $\gtp$ is a prime ideal of ${\mathcal S}(X)$ and ${\tt d}_X(\gtp)=n$\em.

Indeed, as $\gtp\subset\gtm_0$, it is a proper subset of ${\mathcal S}(X)$. It is clear that the set $\gtp$ is closed under multiplication by elements of ${\mathcal S}(X)$ and let us see that it is closed under addition. Indeed, let $f_1,f_2\in\gtp$ and let $(K,\Phi)$ be a semialgebraic triangulation of $X$ compatible with $Z(f_1)$ and $Z(f_2)$. Let $\sigma_i$ be an indicator simplex of $f_i$. For our purposes it is enough to check that $\sigma_1=\sigma_2$. But this follows recursively using the following straightforward fact: 

\vspace{2mm}\setcounter{substep}{0}
\begin{substeps}{excrucial}\label{propcrucial}
\em Let $\tau\subset R^d$ be simplex of dimension $d$ and let $\eta_1,\eta_2$ be two simplices contained in $R^d\times[0,\infty)$ that have $\tau$ as a common face. Then $\eta_1^0\cap\eta_2^0\neq\varnothing$.
\end{substeps} 

\vspace{2mm}
Next, we prove that the ideal $\gtp$ is prime, being clear that ${\tt d}_X(\gtp)=n$. Indeed, let $f_1,f_2\in{\mathcal S}(X)$ be such that $f_1f_2\in\gtp$ and let $(K,\Phi)$ be a semialgebraic triangulation of $X$ compatible with $Z(f_1)$ and $Z(f_2)$. Let $\sigma$ be an indicator simplex of $f_1f_2$. Using (\ref{excrucial}.\ref{propcrucial}) recursively, one shows that the condition \em for each $d=0,\ldots,n$ there is a $d$-dimensional face $\tau_d$ of $\sigma$ so that $\Phi(\tau_d)\subset\{\x_{d+1}=0,\ldots,\x_n=0\}$ \em determines uniquely $\sigma$. Since $\Phi(\sigma)\subset Z(f_1f_2)$ and $(K,\Phi)$ is compatible with $Z(f_i)$, we may assume that $\Phi(\sigma^0)\subset Z(f_1)$; hence, $\Phi(\sigma)\subset Z(f_1)$ and so $f_1\in\gtp$. We conclude that $\gtp$ is a prime ideal, as required.

\vspace{2mm}
(ii) Write $X_n:=[0,1]^n$. We claim that: \em There is a chain of prime ideals $\gtq_0\subsetneq\cdots\subsetneq\gtq_n:=\gtm_0$ in ${\mathcal S}(X_n)$ such that ${\tt d}_X(\gtq_k)=n-k$ for $k=0,\ldots,n$\em.

For each $k=1,\ldots,n$ define $X_k:=[0,1]^k\times\{0\}\subset R^n$. Clearly, $\{0\}\subsetneq X_1\subsetneq\cdots\subsetneq X_n$ is a chain of closed subsets of $X_n$. The restriction homomorphism $\varphi_k:{\mathcal S}(X_n)\to{\mathcal S}(X_k),\ f\mapsto f|_{X_k}$ is, by \cite{dk}, surjective and so the ideal $\gtp_k$ constructed in (i) for $X_k$ provides a prime ideal $\gtq_{n-k}:=\varphi_k^{-1}(\gtp_k)$ such that ${\tt d}_{X}(\gtq_{n-k})={\tt d}_{X_k}(\gtp_k)=k$. Now, by the very definition of the ideals $\gtq_k$, it is clear that $\gtq_0\subsetneq\cdots\subsetneq\gtq_n:=\gtm_0$.
\end{example}

\subsection{Proof of the main results}
We are ready to approach the proofs of Theorems \ref{dimension} and \ref{localdimension}.

\begin{proof}[Proof of Theorem \em\ref{dimension}]
Combining (\ref{cita}.\ref{inequalities}) and Corollary \ref{ht} it holds that 
\begin{equation*}
\dim{\mathcal S}(M)\leq\dim{\mathcal S}^*(M)\leq\sup_{(X,{\tt j})\in{\mathfrak F}_M}\{\dim{\mathcal S}(X)\}\leq\dim M.
\end{equation*}
To finish we must show that $\dim M\leq\dim{\mathcal S}(M)$, but this follows from Theorem \ref{localdimension}, that we prove next.
\end{proof}

\begin{proof}[Proof of Theorem \em\ref{localdimension}]
First, by Lemma \ref{subset} there is a semialgebraic embedding $h:[0,1]^d\hookrightarrow M$ such that $h(0)=p$; denote $T:=\im(h)$. By Example \ref{excrucial}, there is a chain of prime ideals $\gtq_0\subsetneq\cdots\subsetneq\gtq_d:=\gtnd_p$ in ${\mathcal S}(T)$. The surjective homomorphism $\varphi:{\mathcal S}^{\diam}(M)\to{\mathcal S}(T),\ f \mapsto f|_T$ (see \cite{dk}) provides a chain of prime ideals $\gtp_0\subsetneq\cdots\subsetneq\gtp_d:=\gtmd_p$ in ${\mathcal S}^{\diam}(M)$ where each $\gtp_k:=\varphi^{-1}(\gtq_k)$; hence, $d\leq\hgt(\gtmd_p)$.

For the converse inequality $\hgt(\gtmd_p)\leq d$, let $\gtp_0'\subsetneq\cdots\subsetneq\gtp_\ell'=\gtmd_p$ be a chain of prime ideals in ${\mathcal S}(M)$. By (\ref{cita}.\ref{longitud}), there exist a semialgebraic compactification $(X,{\tt j})$ of $M$ and a chain of prime ideals $\gtP_0'\subsetneq\cdots\subsetneq\gtP_{\ell}'$ of ${\mathcal S}(X)$ with $\gtP_i':=\gtp_i'\cap{\mathcal S}(X)$. After identifying $M$ with ${\tt j}(M)$, we write $X=\cl(M)$. Observe that $\gtP_\ell'=\gtmd_p\cap{\mathcal S}(X)$ is the maximal ideal consisting of those semialgebraic functions of ${\mathcal S}(X)$ vanishing at $p$. Thus, by Lemma \ref{ht0}, we have $0={\tt d}_X(\gtP_\ell')<\cdots<{\tt d}_X(\gtP_0')$; hence, $\ell\leq{\tt d}_X(\gtP_0')$. Let $r>0$ be small enough so that the open ball $\Bb$ of center $p$ and radius $r$ satisfies $\dim(X\cap\cl(\Bb))=\dim(M\cap \Bb)=d$. Consider the closed semialgebraic sets defined as $T_1:=\cl(\Bb)$ and $T_2:= R^n\setminus \Bb$. Let $f_1,f_2\in{\mathcal S}^*(R^n)$ be such that $Z_{R^n}(f_i)=T_i$. The product $f_1f_2$ vanishes identically on $R^n$, hence on $X$. Thus, $(f_1|_X)(f_2|_X)\in\gtP_0'$ and since $f_2(p)\neq0$, we conclude that $f_1|_X\in\gtP_0'$; hence
$$
\ell\leq{\tt d}_X(\gtP_0')\leq\dim(Z_X(f_1|_X))=\dim(X\cap\cl(\Bb))=d.
$$
Therefore, $\hgt(\gtmd_p)\leq d$ as wanted.

To finish notice that, in case ${\mathcal S}^{\diam}(M)={\mathcal S}(M)$, the prime ideals $\gtp_k$ constructed above are $z$-ideals, as they are the inverse images of the prime $z$-ideals $\gtq_k$ by the surjective homomorphism ${\mathcal S}(M)\to{\mathcal S}(T),\ f \mapsto f|_T$. Moreover, it follows from Theorem \ref{rmdc} and Example \ref{excrucial} that
$$
\tr\deg_R(\qf({\mathcal S}^{\diam}(M)/\gtp_k))=\tr\deg_R(\qf({\mathcal S}(T)/\gtq_k))={\tt d}_T(\gtq_k)=d-k,
$$ 
as required.
\end{proof}

\bibliographystyle{amsalpha}

\end{document}